\documentclass[12pt,a4paper]{amsart}
\usepackage{amsfonts}
\usepackage{amsthm}
\usepackage{amsmath}
\usepackage{amscd}
\usepackage[latin2]{inputenc}
\usepackage{t1enc}
\usepackage[mathscr]{eucal}
\usepackage{indentfirst}
\usepackage{graphicx}
\usepackage{graphics}
\usepackage{pict2e}
\usepackage{epic}
\numberwithin{equation}{section}
\usepackage[margin=2.9cm]{geometry}
\usepackage{epstopdf} 
\usepackage[sc]{mathpazo}    
\usepackage[scaled]{helvet}  
\usepackage{eulervm}

\theoremstyle{plain}
\newtheorem{Th}{Theorem}[section]
\newtheorem{Lemma}[Th]{Lemma}
\newtheorem{Cor}[Th]{Corollary}
\newtheorem{Prop}[Th]{Proposition}

 \theoremstyle{definition}

\newtheorem{Rem}[Th]{Remark}
\newtheorem{?}[Th]{Problem}

\begin{document}

\title{$\beta$-High Resolution ODE and Phase Transition between NAG-SC and Heavy Ball Method}

\author[Da Wu]{Da Wu}

\address{University of Pennsylvania \\ Department of Mathematics\\ David Rittenhouse Lab \\209 South 33rd Street\\ Philadelphia, PA 19104-6395} 

\email{dawu@math.upenn.edu}

 \subjclass[2010]{}


\begin{abstract}
In this paper, we study the convergence properties of an algorithm that can be viewed as an interpolation between two gradient based optimization methods, Nesterov's acceleration method for strongly convex functions $(NAG$-$SC)$ and Polyak's heavy ball method. Recent Progress \cite{HRODE} has been made on using \textit{High-Resolution} ordinary differential equations (ODEs) to distinguish these two fundamentally different methods. The key difference between them can be attributed to the \textit{gradient correction term}, which is reflected by the Hessian term in the High-Resolution ODE. Our goal is to understand how this term can affect the convergence rate and the choice of our step size. To achieve this goal, we introduce the notion of \textit{$\beta$-High Resolution ODE}, $0\leq \beta\leq 1$ and prove that within certain range of step size, there is a phase transition happening at $\beta_c$. When $\beta_c\leq\beta\leq 1$, the algorithm associated with $\beta$-High Resolution ODE have the same convergence rate as NAG-SC. When $0\leq \beta\leq \beta_c$, this algorithm will have the slower convergence rate than NAG-SC. 
\end{abstract}

\maketitle
\tableofcontents
\section{Introduction}
\subsection{Overview}
In modern machine learning and (convex) optimization, we are interested in efficiently finding the minimizer of a smooth convex function $f:\mathbb R^n\to \mathbb R$, i.e.
\begin{equation}
	\min_{x\in \mathbb R^n} f(x)
\end{equation}
 There are several ways of solving this unconstrained optimization problem, among which the simplest and most straightforward method is \textit{gradient descent}. For any initial point $x_0\in \mathbb R^n$, we update our $x_k$ by the following recursive rule, 
\begin{equation}\label{gradient descent}
	x_{k+1}=x_k-s\nabla f(x_k)
\end{equation}
where $s>0$ is a fixed step size. Significant amount of work has been devoted to improve $(\ref{gradient descent})$ afterwards. Polyak in \cite{Pol1}, \cite{Pol2} introduced the following \textit{heavy ball method}. For any two initial points $x_0,x_1\in \mathbb R^n$, we iteratively update our $x_k$ by 
\begin{equation}\label{heavy ball method}
	x_{k+1}=x_k+\alpha(x_k-x_{k-1})-s\nabla f(x_k)
\end{equation}
where $s>0$ is the step size, $\alpha>0$ is called the \textit{momentum coefficient}. Heuristically, at each step, we accelerate the minimizing process by giving a momentum from the previous two steps. The main advantage of this method is the faster local convergence rate near the minimum of $f$. 
\vskip0.05cm%
It turns out that we can do better. Nesterov discovered the \textit{accelerated gradient method}, see \cite{Nest1}, \cite{Nest2} for details.  For (weakly) convex function $f$ (called NAG-C), NAG-C takes the form
\begin{align}
	\begin{split}
		y_{k+1} &=x_k-s\nabla f(x_k)\\
		x_{k+1} &=y_{k+1}+\frac{k}{k+3}(y_{k+1}-y_k)
	\end{split}
\end{align}
 with $x_0=y_0\in \mathbb R^n$.For $\mu$-strongly convex and $L$-Lipschitz function $f$ (called NAG-SC), NAG-SC takes the following form
\begin{align}\label{original NAG-SC}
\begin{split}
	y_{k+1} &=x_k-s\nabla f(x_k)\\
	x_{k+1} &=y_{k+1}+\frac{1-\sqrt{\mu s}}{1+\sqrt{\mu s}}(y_{k+1}-y_k)
\end{split}
\end{align}
with $x_0=y_0\in \mathbb R^n$ as initial data points.(all the terms above will be defined in the next section) Plugging the $y_k$ and $y_{k+1}$ into the second line and we get 
\begin{equation}\label{single variable formula for NAG-SC}
	x_{k+1}=x_k+\left(\frac{1-\sqrt{\mu s}}{1+\sqrt{\mu s}}\right)(x_k-x_{k-1})-s\nabla f(x_k)-\left(\frac{1-\sqrt{\mu s}}{1+\sqrt{\mu s}}\right)s\left(\nabla f(x_k)-\nabla f(x_{k-1}) \right)
\end{equation}
 with $x_0$ and $x_1=x_0-\frac{2s\nabla f(x_0)}{1+\sqrt{\mu s}}$. If we compare $(\ref{single variable formula for NAG-SC})$ with $(\ref{heavy ball method})$, $(\ref{single variable formula for NAG-SC})$ is just the $(\ref{heavy ball method})$ with momentum coefficient $\alpha=\frac{1-\sqrt{\mu s}}{1+\sqrt{\mu s}}$ and an additional term 
 \begin{equation}\label{gradient correction term}
 	\left(\frac{1-\sqrt{\mu s}}{1+\sqrt{\mu s}}\right)s\left(\nabla f(x_k)-\nabla f(x_{k-1}) \right)
 \end{equation}
 This term is called the \textit{gradient correction term}. Mathematically, we want to understand why this term $(\ref{gradient correction term})$ gives a faster convergence rate.
 \vskip0.1cm%
 Recently, the work of B.Shi, S.Du, M.Jordan and W.Su \cite{HRODE} provides an \textit{High-Resolution ODE} approach to unravel the mystery of the gradient correction term. The crucial point in their approach is that when deriving the ODE, we take the step size $s$ small but non-vanishing. Here, we recall that High-Resolution ODE of heavy-ball method is 
 \begin{equation}\label{HR ODE for heavy ball method}
 	\ddot X(t)+2\sqrt{\mu}\dot X(t)+(1+\sqrt{\mu s})\nabla f(X(t))=0
 \end{equation}
 and the High-Resolution ODE of NAG-SC is 
 \begin{equation}\label{HR ODE for NAG-SC}
 	\ddot X(t)+2\sqrt{\mu}\dot X(t)+\sqrt{s}\nabla^2 f(X(t))\dot X(t)+(1+\sqrt{\mu s})\nabla f(X(t))=0
 \end{equation}
 If we simply take the step size $s\to 0$, then both heavy ball method and NAG-SC will have the same limiting ODE (see \cite{DiffM} and \cite{HRODE} for a more detailed discussion)
 \begin{equation}
 	\ddot X(t)+2\sqrt{\mu}\dot X(t)+\nabla f(X(t))=0
 \end{equation}
We can see that the only difference between $(\ref{HR ODE for heavy ball method})$ and $(\ref{HR ODE for NAG-SC})$ is the $\sqrt{s}\nabla^2 f(X(t))\dot X(t)$. In order to better understand how this term would make a difference on convergence rate and step size, we consider the so-called $\beta$ High-Resolution ODE, 
\begin{equation}
	\ddot X(t)+2\sqrt{\mu}\dot X(t)+\beta\sqrt{s}\nabla^2 f(X(t))\dot X(t)+(1+\sqrt{\mu s})\nabla f(X(t))=0,\qquad 0\leq\beta\leq 1
\end{equation}
Its corresponding discrete counterpart 
\begin{align*}
		y_{k+1} &=x_k-s\nabla f(x_k)\nonumber \\
		y_{k+1}^\beta &=x_k-\beta s\nabla f(x_k)\\
		x_{k+1} &=y_{k+1}+\frac{1-\sqrt{\mu s}}{1+\sqrt{\mu s}}\left(y_{k+1}^{\beta}-y_{k}^{\beta} \right)\nonumber
	\end{align*}
can be viewed as an interpolation between NAG-SC and heavy ball method. (see Section $2$ for a detailed derivation) 
\vskip0.1cm%
The main objective of this paper is to understand the "cutoff" point of the convergence rate of this generalized class of algorithm when $\beta$ continuously vary from $1$ to $0$. Suppose $\beta$ is negligible, the Hessian term only contributes a little "acceleration". Hence it cannot achieve the same convergence rate as NAG-SC. Similarly, suppose $\beta$ is very close to $1$, it is essentially NAG-SC, which should give us a faster convergence rate than heavy ball method. To start, we first introduce some basic definitions.       \subsection{Notation and Basic Setup}
Let $\mathcal F^1_L(\mathbb R^n)$ denote the class of $L$-smooth convex functions defined on $\mathbb R^n$, that is, $f\in \mathcal F_L^1$ if $f(y)\geq f(x)+\langle \nabla f(x), y-x\rangle$ for all $x,y\in \mathbb R^n$. Its gradient is $L$-Lipschitz continuous in the sense that 
\begin{equation*}
	\|\nabla f(x)-\nabla f(y)\|\leq L\|x-y\|
\end{equation*} 
where $\|\cdot\|$ denotes standard Euclidean norm and $L>0$ is the Lipschitz constant. The function class $\mathcal F^2_L(\mathbb R^n)$ denotes the subclass of $\mathcal F^1_L(\mathbb R^n)$ such that each $f$ has a Lipschitz continuous Hessian in the sense that 
\begin{equation*}
	\left\|\nabla^2 f(x)-\nabla^2 f(y)\right\|_{F}\leq L'\|x-y\|
\end{equation*}
where $\|\cdot\|_F$ denotes the Frobenius norm and $L'>0$ is an arbitrary constant. For $p=1,2$, let $\mathcal S_{\mu,L}^p(\mathbb R^n)$ denote the subclass of $\mathcal F_L^p(\mathbb R^n)$ such that each member $f$ is $\mu$-strongly convex for some $0<\mu\leq L$. That is, $f\in \mathcal S_{\mu,L}^p(\mathbb R^n)$ if $f\in \mathcal F_L^p(\mathbb R^n)$ and 
\begin{equation*}
	f(y)\geq f(x)+\langle \nabla f(x),y-x\rangle+\frac{\mu}{2}\|y-x\|^2
\end{equation*} 
for all $x,y\in \mathbb R^n$. This is equivalent to the convexity of $f(x)-\frac{\mu}{2}\|x-x^*\|^2$, where $x^*$ is the minimizer of the objective $f$. Now, we are ready to state the main result.
\subsection{Statement of the Main Result}
\begin{Th}\label{Main Theorem}
	Let $f\in \mathcal S_{\mu,L}^{1}(\mathbb R^n)$. If the step size $s$ satisfies $\frac{25\mu}{(12L-\mu)^2} \leq s=\frac{1}{cL} \leq \frac{1}{4L}\ \ (\text{equivalently,}\  4\leq c\leq \frac{(12L-\mu)^2}{25\mu L}
	 , \text{$c$ may possibly depend on $\mu,L$})$, then there exists a $\beta_c=\beta_c(\mu,L,s)\in [0,1)$ such that when $0\leq\beta\leq \beta_c$, 
	\begin{align}\label{Convergence Rate for beta-High Resolution method}
		f(x_k)-f(x^*)\leq O\left( \frac{L\cdot\|x_0-x^*\|^2 }{\left\lbrace 1+\frac{\frac{\beta^2-\beta}{c^2}\left(\frac{\mu}{L}\right)+\left(\frac{1}{\sqrt{c}}-\frac{3+\beta^2-2\beta}{c\sqrt{c}}\right)\sqrt{\frac{\mu}{L}}-\frac{2-2\beta}{c}}{\frac{\beta^2}{2c^2\sqrt{c}}\left(\frac{\mu}{L}\right)^{3/2}-\left(\frac{1}{c}+\frac{\beta^2}{c^2}\right)\frac{\mu}{L}+\left(\frac{1}{\sqrt{c}}+\frac{\beta^2}{2c\sqrt{c}}\right)\sqrt{\frac{\mu}{L}} }\right\rbrace^k}\right)
	\end{align}
	when $\beta_c\leq \beta\leq 1$, 
	\begin{align}\label{previous result}
		f(x_k)-f(x^*)\leq O\left(\frac{L\cdot\|x_0-x^*\|^2}{\left\lbrace 1+\frac{1}{6\sqrt{c}}\sqrt{\frac{\mu}{L}}\right\rbrace^k} \right)
	\end{align}
	$\beta_c$ is computed explicitly in Remark $\ref{Critical Value}$. 
\end{Th}
\begin{Rem}[Comparison with the known results]
	In \cite{HRODE} , Theorem $3$, when $s=\frac{1}{4L}$, NAG-SC $(\beta=1)$ gives us a monotone convergence rate of 
	\begin{align}\label{rate of convergence for NAG-SC}
		f(x_k)-f(x^*)\leq O\left(\frac{L\cdot\|x_0-x^*\|^2}{\left(1+\frac{1}{12}\sqrt{\mu/L} \right)^k} \right)
	\end{align}
	In \cite{DiffM} , Theorem $4$, if the step size $s$ is set to be $s=\frac{\mu}{16L^2}$, then the Heavy Ball Method $(\beta=0)$ gives us a monotone convergence rate of 
	\begin{align}\label{rate of convergence for heavy ball method}
		f(x_k)-f(x_0)\leq O \left(\frac{L\cdot \|x_0-x^*\|^2}{\left(1+\frac{\mu}{16L}\right)^k} \right)
	\end{align}
	In our $\beta$-High Resolution Approach, assume $s\propto\frac{1}{L}$, we can see that as $\beta$ decreases from $1$ to $0$, after passing the critical value $\beta_c$, the convergence rate cannot match the $(\ref{rate of convergence for NAG-SC})$ anymore (It slows down). Instead, the denominator is a rational function of $\sqrt{\mu/L}$ as in $(\ref{Convergence Rate for beta-High Resolution method})$. 
\end{Rem}
\section{Derivation of $\beta$-High Resolution ODE}
For variable $\beta\in [0,1]$, define the $\beta$ generalized NAG-SC method to be 
\begin{align}\label{beta discrete method}
		y_{k+1} &=x_k-s\nabla f(x_k)\nonumber \\
		y_{k+1}^\beta &=x_k-\beta s\nabla f(x_k)\\
		x_{k+1} &=y_{k+1}+\frac{1-\sqrt{\mu s}}{1+\sqrt{\mu s}}\left(y_{k+1}^{\beta}-y_{k}^{\beta} \right)\nonumber
	\end{align} 
with initial condition $x_0\in \mathbb R^n$ and $y_0^\beta=\frac{(1-\sqrt{\mu s})x_0-s\nabla f(x_0)\left[(1-\sqrt{\mu s})\beta+\sqrt{\mu s}-1\right]}{1-\sqrt{\mu s}}$. This is equivalent to 
\begin{equation}\label{Discrete Method}
		x_{k+1}=x_k+\frac{1-\sqrt{\mu s}}{1+\sqrt{\mu s}}(x_k-x_{k-1})-s\nabla f(x_k)-\beta\cdot \frac{1-\sqrt{\mu s}}{1+\sqrt{\mu s}}\cdot s(\nabla f(x_k)-\nabla f(x_{k-1}))
	\end{equation}
with initial condition $x_0$ and $x_1=x_0-\frac{2s\nabla f(x_0)}{1+\sqrt{\mu s}}$. Fix a nonnegative integer $k$ and let $t_k=k\sqrt{s}$ and $x_k=X(t_k)$ for some $C^\infty$ curve. Using Taylor expansion with respect to $\sqrt{s}$, we get 
\begin{align}
	x_{k+1} &=X(t_{k+1})=X(t_k)+\dot X(t_k)\sqrt{s}+\frac{1}{2}\ddot X(t_k)\left(\sqrt{s}\right)^2+\frac{1}{6}\dddot X(t_k)\left(\sqrt{s}\right)^3+O\left(\left(\sqrt{s}\right)^4\right)\label{Taylor eq 1} \\
	x_{k-1} &=X(t_{k-1})=X(t_k)-\dot X(t_k)\sqrt{s}+\frac{1}{2}\ddot X(t_k)\left(\sqrt{s}\right)^2-\frac{1}{6}\dddot X(t_k)\left(\sqrt{s}\right)^3+O\left(\left(\sqrt{s}\right)^4\right)\label{Taylor eq 2}
\end{align}
Applying Taylor expansion again to the gradient correction gives us 
\begin{equation}\label{Gradient Difference}
	\nabla f(x_k)-\nabla f(x_{k-1})=\nabla^2 f(X(t_k))\dot X(t_k)\sqrt{s}+O\left(\left(\sqrt{s}\right)^2 \right)
\end{equation}
Multiplying both sides of $(\ref{Discrete Method})$ by $\frac{1+\sqrt{\mu s}}{1-\sqrt{\mu s}}\cdot \frac{1}{s}$ and rearranging the terms, 
	\begin{equation}\label{Single variable discrete}
		\frac{x_{k+1}+x_{k-1}-2x_k}{s}+\frac{2\sqrt{\mu s}}{1-\sqrt{\mu s}}\frac{x_{k+1}-x_k}{s}+\beta\left(\nabla f(x_k)-\nabla f(x_{k-1})\right)+\frac{1+\sqrt{\mu s}}{1-\sqrt{\mu s}}\nabla f(x_k)=0
	\end{equation}
	Plugging $(\ref{Taylor eq 1})$, $(\ref{Taylor eq 2} )$ and $(\ref{Gradient Difference})$ into $(\ref{Single variable discrete})$, we have 
	\begin{align*}
		\ddot X(t_k)+O\left(\left(\sqrt{s}\right)^2\right) &+\frac{2\sqrt{\mu}}{1-\sqrt{\mu s}}\left[\dot X(t_k)+\frac{1}{2}\ddot X(t_k)\sqrt{s}+O\left(\left(\sqrt{s}\right)^2\right) \right]\\
	&+\beta \nabla^2 f(X(t_k))\dot X(t_k)\sqrt{s}+O\left(\left(\sqrt{s}\right)^2 \right)	+\frac{1+\sqrt{\mu s}}{1-\sqrt{\mu s}}\nabla f(X(t_k))=0
	\end{align*}
	After rearranging, 
	\begin{equation*}
		\frac{\ddot X(t_k)}{1-\sqrt{\mu s}}+\frac{2\sqrt{\mu}}{1-\sqrt{\mu s}}\dot X(t_k)+\beta \sqrt{s}\nabla^2 f(X(t_k))\dot X(t_k)+\frac{1+\sqrt{\mu s}}{1-\sqrt{\mu s}}\nabla f(X(t_k))+O(s)=0
	\end{equation*}
	Multiplying both sides by $1-\sqrt{\mu s}$ and by ignoring any $O(s)$ terms but keep $O(\sqrt{s})$ terms, we finally get the $\beta$-High Resolution ODE, 
	\begin{equation}\label{beta}
		\ddot{X}(t)+2\sqrt{\mu}\dot{X}(t)+\beta\sqrt{s}\nabla^2 f(X(t))\dot{X}(t)+(1+\sqrt{\mu s})\nabla f(X(t))=0
	\end{equation}
with $0\leq \beta\leq 1$. The initial conditions of $(\ref{beta})$ throughout this paper are assumed to be $X(0)=x_0$ and $\dot X(0)=-\frac{2\sqrt{s}\nabla f(x_0)}{1+\sqrt{\mu s}}$.\section{Global Existence and Uniqueness of ODE}
Suppose $X_s(t)$ is the solution of $(\ref{beta})$, then by the following Lyapunov function
\begin{equation}\label{Lyapunov Function}
	\mathcal E(t)=(1+\sqrt{\mu s})(f(X_s)-f(x^*))+\frac{1}{2}\left\|\dot X_s\right\|^2
\end{equation}
we can deduce that there exists some $\mathcal C_1>0$ such that
\begin{equation*}
	\sup_{0\leq t<\infty}\left\|\dot X_s(t) \right\|\leq \mathcal C_1
\end{equation*} 
Now, we investigate the global existence and uniqueness of the $\beta$-High Resolution ODE $(\ref{beta})$. Recall that the initial value problem (IVP) for first-order ODE system in $\mathbb R^m$ is 
\begin{equation}\label{First Order ODE}
	\dot x=b(x),\qquad x(0)=x_0
\end{equation}
and the following theorem deals with the global existence and uniqueness of (\ref{First Order ODE})
\begin{Th}[Chillingworth \cite{Perko} , Chapter 3.1, Theorem $4$]
Let $M\in \mathbb R^m$ be a compact manifold and $b\in C^1(M)$. If the vector fields $b$ satisfies the global Lipschitz condition
\begin{equation*}
	\|b(x)-b(y)\|\leq \mathcal L\|x-y\|
\end{equation*}
for all $x,y\in M$. Then for any $x_0\in M$, the IVP $(\ref{First Order ODE})$ has a unique solution $x(t)$ defined for all $t\in \mathbb R$. 
\end{Th}
\begin{Th}
	For any $f\in \mathcal S_\mu^2(\mathbb R^n):=\cup_{L\geq \mu}\mathcal S^2_{\mu,L}(\mathbb R^n)$, the $\beta$-High Resolution ODE $(\ref{beta})$ with the specified initial conditions has a unique global solution $X\in C^2(I;\mathbb R^n)$.
\end{Th}
\begin{proof}
	Notice that 
	\begin{equation*}
		M_{\mathcal C_1}:=\left\lbrace(X_s,\dot X_s)\in \mathbb R^{2n}\big|\left\|\dot X_s \right\|\leq \mathcal C_1  \right\rbrace
	\end{equation*}
	is a compact manifold. The phase-space representation for $(\ref{beta})$ is
	\begin{equation}\label{Phase space representation}
		\frac{d}{dt}\begin{pmatrix}
			X_s\\
			\dot X_s
		\end{pmatrix}
		=
		\begin{pmatrix}
			\dot X_s\\
			-2\sqrt{\mu}\dot X_s-\beta\sqrt{s}\nabla^2 f(X_s)\dot X_s-(1+\sqrt{\mu s})\nabla f(X_s)
		\end{pmatrix}
	\end{equation}
	Now, for any $\begin{pmatrix}
		X_s,\dot X_s
	\end{pmatrix}^\top$, $\begin{pmatrix}
		Y_s,\dot Y_s
	\end{pmatrix}^\top\in  M_{\mathcal C_1}$, 
	\begin{align*}
		\left\|\frac{d}{dt}\begin{pmatrix}
			X_s\\
			\dot X_s
		\end{pmatrix}-\frac{d}{dt}\begin{pmatrix}
			Y_s\\
			\dot Y_s
		\end{pmatrix} \right\| &\leq \left\|\begin{pmatrix}
			\dot X_s-\dot Y_s\\
			-(2\sqrt{\mu}I+\beta\sqrt{s}\nabla^2 f(X_s))(\dot X_s-\dot Y_s)
		\end{pmatrix} \right\|\\
		&+\beta\sqrt{s}\left\|\begin{pmatrix}
			0\\
			\left(\nabla^2 f(X_s)-\nabla^2 f(Y_s) \right)\dot Y_s
		\end{pmatrix} \right\|\\
		&+(1+\sqrt{\mu s})\left\|\begin{pmatrix}
			0\\
			\nabla f(X_s)-\nabla f(Y_s)
		\end{pmatrix} \right\|\\
		&\leq \sqrt{1+8\mu+2\beta^2sL^2}   \left\|\dot X_s-\dot Y_s \right\| +\left[\beta\sqrt{s}\mathcal C_1 L'+(1+\sqrt{\mu s})L\right]\|X_s-Y_s\|\\
		&\leq 2\max\left\lbrace \sqrt{1+8\mu+2\beta^2sL^2},\beta\sqrt{s}\mathcal C_1 L'+(1+\sqrt{\mu s})L \right\rbrace\left\|
		\begin{pmatrix}
			X_s\\
			\dot X_s
		\end{pmatrix}
		-
		\begin{pmatrix}
			Y_s\\
			\dot Y_s
		\end{pmatrix}
		\right\|
	\end{align*} 
	Hence,  based on the above calculation and the the phase space representation $(\ref{Phase space representation})$, we get the desired results. 
\end{proof}
Here we quickly remark that the low resolution counterparts of this $\beta$-High Resolution ODE is the same as both of the heavy-ball method and NAG-SC, which is 
\begin{equation}\label{Low Resolution ODE}
	\ddot X(t)+2\sqrt{\mu}\dot X(t)+\nabla f(X(t))=0
\end{equation}
 Based on the Lyapunov function $(\ref{Lyapunov Function})$, the gradient norm is also bounded, i.e. 
 \begin{equation*}
 	\sup_{0\leq t<\infty} \left\|\nabla f(X_s(t)) \right\|\leq \mathcal C_2
 \end{equation*}
 For the low resolution ODE $(\ref{Low Resolution ODE})$, it has phase representation 
 \begin{equation}\label{Phase Representation for low resolution}
 	\frac{d}{dt}
 	\begin{pmatrix}
 		X\\
 		\dot X
 	\end{pmatrix}
 	=
 	\begin{pmatrix}
 		\dot X\\
 		-2\sqrt{\mu}\dot X-\nabla f(X)
 	\end{pmatrix}
 \end{equation}
 and again by Lyapunov function, the solution $X=X(t)$ of $(\ref{Low Resolution ODE})$ is bounded, i.e.
 \begin{equation*}
 	\sup_{0\leq t<\infty}\left\|\dot X(t) \right\|\leq\mathcal C_3
 \end{equation*}
 It is easy to see that we can find a constant $\mathcal L_1$ such that
 \begin{equation*}
 	\left\|
 	\begin{pmatrix}
 		\dot X\\
 		-2\sqrt{\mu}\dot X-\nabla f(X)
 	\end{pmatrix}
 	-
 	\begin{pmatrix}
 		\dot Y\\
 		-2\sqrt{\mu}\dot Y-\nabla f(Y)
 	\end{pmatrix}
    \right\|
    \leq 
     \mathcal L_1
    \left\|
    \begin{pmatrix}
    	X\\
    	\dot X
    \end{pmatrix}
    -
    \begin{pmatrix}
    	Y\\
    	\dot Y
    \end{pmatrix}
    \right\|
 \end{equation*}
 Now, we study the approximation. We first introduce several lemmas. 
\begin{Lemma}[Gronwall's Lemma]\label{Calculus Lemma}
	Let $m(t), t\in [0,T]$, be a nonnegative function with the following relation,
	\begin{equation*}
		m(t)\leq C+\alpha\int_0^t m(s)ds
	\end{equation*}
	with $C,\alpha>0$. Then we have 
	\begin{equation*}
		m(t)\leq Ce^{\alpha t}
	\end{equation*}
\end{Lemma}
\begin{proof}
	Trivially by calculus.
\end{proof}
\begin{Lemma}\label{Intermediate Lemma}
	Let $X_s(t)$ and $X(t)$ be the solutions of $\beta$-High Resolution ODE $(\ref{beta})$ and Low Resolution Counterpart $(\ref{Low Resolution ODE})$, respectively. Then 
	\begin{equation*}
		\lim_{s\to 0} \max_{0\leq t\leq T}\left\|X_s(t)-X(t)\right\|=0
	\end{equation*}
\end{Lemma}
\begin{proof}
	By $(\ref{Phase space representation})$ and $(\ref{Phase Representation for low resolution})$, 
	\begin{align*}
		\frac{d}{dt}
		\begin{pmatrix}
			X_s-X\\
			\dot X_s-\dot X
		\end{pmatrix}
		=
		\begin{pmatrix}
			\dot X_s-\dot X\\
			-2\sqrt{\mu}(\dot X_s-\dot X)-(\nabla f(X_s)-\nabla f(X))
		\end{pmatrix}
		-\sqrt{s}
		\begin{pmatrix}
			0\\
			\beta \nabla^2 f(X_s)\dot X_s+\sqrt{\mu}\nabla f(X_s)
		\end{pmatrix}
	\end{align*}
	Then, we have 
	\begin{align*}
		&\ \ \ \ \|X_s(t)-X(t)\|^2+\|\dot X_s(t)-\dot X(t)\|^2\\
		&=2\int_0^t\left\langle
		\begin{pmatrix}
			X_s(u)-X(u)\\
			\dot X_s(u)-\dot X(u)
		\end{pmatrix}
		,
		\frac{d}{du}
		\begin{pmatrix}
			X_s(u)-X(u)\\
			\dot X_s(u)-\dot X(u)
		\end{pmatrix}
		 \right\rangle
		 du+\|X_s(0)-X(0)\|^2+\|\dot X_s(0)-\dot X(0)\|^2\\
		 &\leq 2\mathcal L_1\int_0^t \|X_s(u)-X(u)\|^2+\|\dot X_s(u)-\dot X(u)\|^2du\\
		 &\ \ \ +\left[(\mathcal C_1+\mathcal C_3)(\beta L\mathcal C_1+\mathcal C_2\sqrt{\mu})+\frac{4\sqrt{s}}{(1+\sqrt{\mu s})^2}\|\nabla f(x_0)\|^2 \right]\sqrt{s}\\
		 &\leq 2\mathcal L_1\int_0^t \|X_s(u)-X(u)\|^2+\|\dot X_s(u)-\dot X(u)\|^2du+\mathcal C_5\sqrt{s}
	\end{align*}
	By Lemma $(\ref{Calculus Lemma})$, we have that 
	\begin{equation*}
		\|X_s(t)-X(t)\|^2+\left \|\dot X_s(t)-\dot X(t)\right \|^2\leq \mathcal C_5\sqrt{s}\exp(2\mathcal L_1)t
	\end{equation*}
	This completes the proof. 
\end{proof}
\begin{Lemma}\label{previous lemma}
	The discrete method of $\beta$-High Resolution ODE converges to their low-resolution ODE in the sense that 
	\begin{equation*}
		\lim_{s\to 0} \max_{0\leq k\leq \frac{T}{\sqrt{s}}}\|x_k-X(k\sqrt{s})\|=0
	\end{equation*}
\end{Lemma}
\begin{proof}
	The proof of this Lemma follows closely from the method used in \cite{LyapAnalysis} and \cite{DiffM} . Here we do not go into any details.  
\end{proof}
\begin{Prop}
	For any $f\in \mathcal S_\mu^2(\mathbb R^n):=\cup_{L\geq \mu}\mathcal S_{\mu,L}^2(\mathbb R^n)$, the $\beta$-High Resolution ODE $(\ref{beta})$ with the specified initial conditions has a unique global solution $X\in C^2([0,\infty);\mathbb R^n)$. Moreover, the discretized method converges to the $\beta$-High Resolution ODE in the sense that 
	\begin{equation*}
		\limsup_{s\to 0}\max_{0\leq k\leq\frac{T}{\sqrt{s}}}\left \|x_k-X(k\sqrt{s})\right \|=0
	\end{equation*} 
	for any fixed $T>0$. 
\end{Prop}
\begin{proof}
	This result follows from the Lemma $\ref{Calculus Lemma}$, Lemma $\ref{Intermediate Lemma}$ and Lemma $\ref{previous lemma}$.  
\end{proof}
\section{Convergence Rate of Continuous ODE}
In this section, we prove the following theorem
\begin{Th}\label{Theorem of Convergence Rate}
	Let $f\in \mathcal S_{\mu,L}^2(\mathbb R^n)$. Then for any step size $0\leq s\leq 1/L$, the solution $X=X(t)$ of the $\beta$-High Resolution ODE $(\ref{beta})$ satisfies 
	\begin{equation*}
		f(X(t))-f(x^*)\leq \frac{3+(2-\beta)^2}{2s}\left\|x_0-x^* \right\|^2\cdot e^{-\frac{\sqrt{\mu}}{4}t}
	\end{equation*}
\end{Th}
We first define the Energy Functional $\mathcal E_\beta(t)$ of $\beta$-High Resolution ODE as the following:
\begin{equation}\label{Energy functional}
	\mathcal E_\beta(t):=\left(1+\sqrt{\mu s}\right)\left(f(X)-f(x^*)\right)+\frac{1}{4}\left\|\dot X\right\|^2+\frac{1}{4}\left\|\dot X+2\sqrt{\mu}(X-x^*)+\beta\sqrt{s}\nabla f(X) \right\|^2
\end{equation} 
The next lemma is of key importance to us.
\begin{Lemma}
	For any step size $s>0$, the energy functional $(\ref{Energy functional})$ with $X=X(t)$ being the our solution to the $\beta$-High Resolution ODE satisfies 
	\begin{equation}\label{4.2}
		\frac{d\mathcal E_\beta(t)}{dt}\leq -\frac{\sqrt{\mu}}{4}\mathcal E_\beta(t)-\underbrace{\frac{1}{4}\left(\frac{8\beta s\sqrt{\mu}-3s\beta^2\sqrt{\mu}}{4}\left\|\nabla f(X) \right\|^2+2\sqrt{\mu}\|\dot X\|^2+(\sqrt{\mu}+\mu\sqrt{s})(f(X)-f(x^*) \right)}_{:=\Delta_\beta}
	\end{equation}
	In particular,
	\begin{equation}\label{Key step for the Theorem}
		\frac{d\mathcal E_\beta(t)}{dt}\leq -\frac{\sqrt{\mu}}{4}\mathcal E_\beta(t)
	\end{equation}
\end{Lemma}
\begin{proof}
	The energy functional $(\ref{Energy functional})$ together with $(\ref{beta})$ give us 
	\begin{align*}
		\frac{d\mathcal E_\beta(t)}{dt} &=(1+\sqrt{\mu s})\left\langle\nabla f(X),\dot X \right\rangle+\frac{1}{2}\left\langle\dot X,-2\sqrt{\mu}\dot X-\beta\sqrt{s}\nabla^2f(X)\dot X-(1+\sqrt{\mu s})\nabla f(X) \right\rangle\\
		&\qquad +\frac{1}{2}\left\langle\dot X+2\sqrt{\mu}(X-x^*)+\beta\sqrt{s}\nabla f(X),-(1+\sqrt{\mu s})\nabla f(X) \right\rangle\\
		&=-\sqrt{\mu}\left(\left\|\dot X \right\|^2+(1+\sqrt{\mu s})\left\langle\nabla f(X), X-x^* \right\rangle+\frac{\beta s}{2}\left\|\nabla f(X)\right\|^2\right)\\
		&\qquad -\frac{\beta\sqrt{s}}{2}\left(\|\nabla f(X)\|^2+\dot X^T \nabla^2f(X)\dot X \right)\\
		&\leq -\sqrt{\mu}\left(\|\dot X \|^2+(1+\sqrt{\mu s})\left\langle\nabla f(X), X-x^* \right\rangle+\frac{\beta s}{2}\left\|\nabla f(X)\right\|^2\right)
	\end{align*} 
	Also, by $\mu$-strong convexity of $f$, 
	\begin{equation*}
		\left\langle \nabla f(X),X-x^*\right\rangle\geq
		\begin{cases}
			f(X)-f(x^*)+\frac{\mu}{2}\|X-x^*\|^2\\
			\mu\|X-x^*\|^2
		\end{cases}
	\end{equation*}
	This gives us 
	\begin{align*}
		(1+\sqrt{\mu s})\left\langle\nabla f(X),X-x^* \right\rangle &\geq \frac{1+\sqrt{\mu s}}{2}\left\langle\nabla f(X),X-x^*\right\rangle+\frac{1}{2}\langle \nabla f(X),X-x^*\rangle\\
		&\geq \frac{1+\sqrt{\mu s}}{2}\left(f(X)-f(x^*)+\frac{\mu}{2}\|X-x^*\|^2 \right)+\frac{\mu}{2}\|X-x^*\|^2\\
		&\geq \frac{1+\sqrt{\mu s}}{2}\left(f(X)-f(x^*)\right)+\frac{3\mu}{4}\|X-x^*\|^2
	\end{align*}
	Hence, the derivative of Energy Functional can be bounded by 
	\begin{equation}\label{4.3}
		\frac{d\mathcal E_\beta(t)}{dt}\leq -\sqrt{\mu}\left(\frac{1+\sqrt{\mu s}}{2}(f(X)-f(x^*))+\|\dot X\|^2+\frac{3\mu}{4}\|X-x^*\|^2+\frac{\beta s}{2}\|\nabla f(X)\|^2 \right)
	\end{equation}
	Next, by Cauchy-Schwarz inequality, 
	\begin{equation*}
		\left\|2\sqrt{\mu}(X-x^*)+\dot X+\beta\sqrt{s}\nabla f(X)\right\|^2\leq  3\left(4\mu \|X-x^*\|^2+\|\dot X\|^2+\beta^2 s\|\nabla f(X)\|^2 \right)
	\end{equation*}
	from which we can deduce that 
	\begin{equation}\label{4.4}
		\mathcal E_\beta(t)\leq (1+\sqrt{\mu s})\left(f(X)-f(x^*)\right)+\|\dot X\|+3\mu \|X-x^*\|^2+\frac{3s\beta^2}{4}\|\nabla f(X)\|^2
	\end{equation}
	Finally, combining $(\ref{4.3} )$ and $(\ref{4.4})$ and we get the $(\ref{4.2})$. The $(\ref{Key step for the Theorem})$ holds since $\Delta_\beta\geq 0$. (Notice that $0\leq \beta\leq 1$ and $x^*$ is the minimizer)  
\end{proof}
\begin{proof}[Proof of Theorem \ref{Theorem of Convergence Rate}]
	By previous lemma, 
	\begin{equation*}
		\dot{\mathcal E}_\beta(t)\leq  -\frac{\sqrt{\mu}}{4}\mathcal E_\beta(t)\implies \frac{d}{dt}\left(\mathcal E_\beta(t)e^{\frac{\sqrt{\mu}}{4}t}\right)\leq 0\implies \mathcal E_\beta(t)\leq e^{-\frac{\sqrt{\mu}}{4}t}\mathcal E_\beta(0)
	\end{equation*}
	Noticing the initial condition $X(0)=x_0$ and $\dot X(0)=-\frac{2\sqrt{s}\nabla f(x_0)}{1+\sqrt{\mu s}}$, we get 
	\begin{align*}
		f(X)-f(x^*) &\leq e^{-\frac{\sqrt{\mu}}{4}t}\biggl[f(x_0)-f(x^*)+\frac{s}{(1+\sqrt{\mu s})^3}\|\nabla f(x_0)\|^2\\
		&+\frac{1}{4(1+\sqrt{\mu s})}\left\|2\sqrt{\mu}(x_0-x^*)-\frac{2-\beta-\beta\sqrt{\mu s}}{1+\sqrt{\mu s}}\cdot \sqrt{s}\nabla f(x_0)\right\|^2\biggr]
	\end{align*}
	Since $f\in \mathcal S_{\mu,L}^2$, 
	\begin{equation*}
		\|\nabla f(x_0)\|\leq L\|x_0-x^*\|\qquad\text{and}\qquad f(x_0)-f(x^*)\leq \frac{L}{2}\cdot \|x_0-x^*\|^2
	\end{equation*}
	Together with Cauchy-Schwartz inequality, 
	\begin{align*}
		f(X)-f(x^*) &\leq \left[f(x_0)-f(x^*)+\frac{2+(2-\beta-\beta\sqrt{\mu s})^2}{2(1+\sqrt{\mu s})^3}\cdot s\|\nabla f(x_0)\|^2+\frac{2\mu}{1+\sqrt{\mu s}}\|x_0-x^*\|^2\right]e^{-\frac{\sqrt{\mu}}{4}t}\\
		&\leq \left[\frac{L}{2}+\frac{2+(2-\beta-\beta\sqrt{\mu s})^2}{2(1+\sqrt{\mu s})^3}\cdot sL^2+\frac{2\mu}{1+\sqrt{\mu s}}\right]\|x_0-x^*\|^2e^{-\frac{\sqrt{\mu}}{4}t}\\
		&\leq \left[\frac{1}{2}+\frac{2+(2-\beta-\beta\sqrt{\mu s})^2}{2(1+\sqrt{\mu s})^3}+\frac{2\mu s}{1+\sqrt{\mu s}}\right]\cdot\frac{1}{s}\cdot \|x_0-x^*\|^2e^{-\frac{\sqrt{\mu}}{4}t}
	\end{align*}
	Now, by a little bit of analysis, under the assumption $\mu s\leq \mu/L\leq 1$, 
	\begin{equation*}
		\frac{1}{2}+\frac{2+(2-\beta-\beta\sqrt{\mu s})^2}{2(1+\sqrt{\mu s})^3}+\frac{2\mu s}{1+\sqrt{\mu s}}\leq \frac{3+(2-\beta)^2}{2}
	\end{equation*}
	This completes the proof of the Theorem.
\end{proof}
\section{Convergence Rate of discrete method}
\subsection{Discrete Energy Functional}
We first write the $(\ref{Discrete Method})$ as  
\begin{align}\label{Phase Space Representation II}
\begin{split}
	x_k-x_{k-1} &=\sqrt{s}v_{k-1}\\
	v_k-v_{k-1} &=-\frac{2\sqrt{\mu s}}{1-\sqrt{\mu s}}v_k-\beta\sqrt{s}(\nabla f(x_k)-\nabla f(x_{k-1}))-\frac{1+\sqrt{\mu s}}{1-\sqrt{\mu s}}\cdot \sqrt{s}\nabla f(x_k)
\end{split}
\end{align}
in the position variable $x_k$ and the velocity variable $v_k$ that is defined as 
\begin{equation*}
	v_k=\frac{x_{k+1}-x_k}{\sqrt{s}}
\end{equation*}
The initial velocity is 
\begin{equation*}
	v_0=-\frac{2\sqrt{s}}{1+\sqrt{\mu s}}\nabla f(x_0)
\end{equation*}
Next, we construct the $\beta$ discrete-time energy functional 
\begin{align}\label{beta discrete-time energy functional}
	\mathcal E_\beta(k) &=\underbrace{\frac{1+\sqrt{\mu s}}{1-\sqrt{\mu s}} (f(x_k)-f(x^*))}_{\mathbf I}+\underbrace{\frac{1}{4}\|v_k\|^2}_{\mathbf {II}}+\underbrace{\frac{1}{4}\left\|v_k+\frac{2\sqrt{\mu}}{1-\sqrt{\mu s}}(x_{k}-x^*)+\beta\sqrt{s}\nabla f(x_k)\right\|^2}_{\mathbf{III}}\nonumber \\
	&\qquad -\underbrace{\frac{\beta s\|\nabla f(x_k)\|^2}{2(1-\sqrt{\mu s})}}_{\textbf{negative term}}
\end{align}
\subsection{Lemmata}
\begin{Lemma}
	For $f\in \mathcal S_{\mu,L}^{1}(\mathbb R^n)$,
	\begin{align*}
  		\mathcal E_\beta(k)&\leq \left(\frac{1}{1-\sqrt{\mu s}}+\frac{\beta^2 Ls}{2}\right)(f(x_k)-f(x^*))+\frac{1+\sqrt{\mu s}+\mu s}{(1-\sqrt{\mu s})^2}\|v_k\|^2\\
  		&+\frac{3\mu}{(1-\sqrt{\mu s})^2}\|x_{k}-x^*\|^2+\frac{\sqrt{\mu s}}{1-\sqrt{\mu s}}\left[f(x_k)-f(x^*)-\left(\frac{\beta^2 s\sqrt{\mu s}-(\beta^2-\beta )s}{2\sqrt{\mu s}}\right)\|\nabla f(x_k)\|^2 \right]
  	\end{align*}
\end{Lemma}
\begin{proof}
	In the definition of $\beta$ discrete-time energy functional (\ref{beta discrete-time energy functional}), by the Cauchy-Scharwz inequality, we have 
  	\begin{align*}
  	\mathbf{III} &=\frac{1}{4}\left\|v_k+\frac{2\sqrt{\mu}}{1-\sqrt{\mu s}}(x_{k}-x^*)+\beta\sqrt{s}\nabla f(x_k) \right\|^2\\
  	&\leq \frac{3}{4}\left[\left(\frac{1+\sqrt{\mu s}}{1-\sqrt{\mu s}}\right)^2\|v_k\|^2+\frac{4\mu }{(1-\sqrt{\mu s})^2}\|x_{k}-x^*\|^2+\beta^2 s\|\nabla f(x_k)\|^2 \right]
  	\end{align*} 
  	Notice that $\|\nabla f(x_k)\|^2\leq 2L\left(f(x_k)- f(x^*)\right)$,  
  	\begin{align*}
  		\frac{3\beta^2 s}{4}\|\nabla f(x_k)\|^2-\frac{\beta s\|\nabla f(x_k)\|^2}{2(1-\sqrt{\mu s})}&=\frac{\beta^2 s}{4}\|\nabla f(x_k)\|^2+\frac{\beta^2 s}{2}\|\nabla f(x_k)\|^2-\frac{\beta s\|\nabla f(x_k)\|^2}{2(1-\sqrt{\mu s})}\\
  		&\leq \frac{\beta^2 L s}{2}(f(x_k)-f(x^*))-\frac{\beta^2 s\sqrt{\mu s}-(\beta^2-\beta)s}{2(1-\sqrt{\mu s})}\cdot \|\nabla f(x_k)\|^2
  	\end{align*}
  	for $f\in \mathcal S^1_{\mu,L}(\mathbb R^n)$, which gives us the following estimate, 
  	\begin{align*}
  		\mathcal E_\beta(k)&\leq \left(\frac{1}{1-\sqrt{\mu s}}+\frac{\beta^2 Ls}{2}\right)(f(x_k)-f(x^*))+\frac{1+\sqrt{\mu s}+\mu s}{(1-\sqrt{\mu s})^2}\|v_k\|^2\\
  		&+\frac{3\mu}{(1-\sqrt{\mu s})^2}\|x_{k}-x^*\|^2+\frac{\sqrt{\mu s}}{1-\sqrt{\mu s}}\left[f(x_k)-f(x^*)-\left(\frac{\beta^2 s\sqrt{\mu s}-(\beta^2-\beta )s}{2\sqrt{\mu s}}\right)\|\nabla f(x_k)\|^2 \right]
  	\end{align*}
\end{proof}
\begin{Lemma}\label{Estimate on the difference of energy functional}
	For $f\in\mathcal S_{\mu,L}^1(\mathbb R^n)$, 
	\begin{align*}
		\mathcal E_\beta(k+1)-\mathcal E_\beta(k) &\leq -\frac{\sqrt{\mu s}}{1-\sqrt{\mu s}}\left(\frac{1+\sqrt{\mu s}}{1-\sqrt{\mu s}}\cdot\langle \nabla f(x_{k+1}), x_{k+1}-x^*\rangle+\|v_{k+1}\|^2\right)\\
		&\ \ \ \ +\frac{1}{2}\left(\frac{1+\sqrt{\mu s}}{1-\sqrt{\mu s}}\right)\cdot s\cdot \frac{(1+\beta)\sqrt{\mu s}+(1-\beta)}{1-\sqrt{\mu s}}\cdot \|\nabla f(x_{k+1})\|^2\\
		&\ \ \ \ -\frac{1}{2L}\left(\frac{\beta-\beta\sqrt{\mu s}}{1+\sqrt{\mu s}}+\frac{1+\sqrt{\mu s}}{1-\sqrt{\mu s}}\right)\|\nabla f(x_{k+1})-\nabla f(x_k)\|^2\\
		&\ \ \ \ +\frac{\beta s}{2}\left(\frac{1+\sqrt{\mu s}}{1-\sqrt{\mu s}}+\frac{1-\sqrt{\mu s}}{1+\sqrt{\mu s}} \right)\| \nabla f(x_{k+1})-\nabla f(x_k)\|^2\\ 
		\end{align*}
	\end{Lemma}
\begin{proof}
	The proof of this Lemma is only a slight variation of argument in \cite{HRODE} , Appendix B.2.2 so here we only give the first several steps in order to illustrate the difference.  Recall the $\beta$ discrete time energy functional $(\ref{beta discrete-time energy functional})$ 
\begin{align*}
	\mathcal E_\beta(k) &=\underbrace{\frac{1+\sqrt{\mu s}}{1-\sqrt{\mu s}} (f(x_k)-f(x^*))}_{\mathbf I}+\underbrace{\frac{1}{4}\|v_k\|^2}_{\mathbf {II}}+\underbrace{\frac{1}{4}\left\|v_k+\frac{2\sqrt{\mu}}{1-\sqrt{\mu s}}(x_{k}-x^*)+\beta\sqrt{s}\nabla f(x_k)\right\|^2}_{\mathbf{III}}\nonumber \\
	&\qquad -\underbrace{\frac{\beta s\|\nabla f(x_k)\|^2}{2(1-\sqrt{\mu s})}}_{\textbf{negative term}}
\end{align*}
Let $\Delta_{\mathbf I}, \Delta_{\mathbf{II}}$ and $\Delta_{\mathbf{III}}$ be the difference between $\mathbf{I},\mathbf{II}$ and $\mathbf{III}$ respectively. For the first part, same as in \cite{HRODE} , Appendix $B.2.2$
\begin{align*}
	\Delta_{\mathbf{I}} &=\frac{1+\sqrt{\mu s}}{1-\sqrt{\mu s}} (f(x_{k+1})-f(x^*))-\frac{1+\sqrt{\mu s}}{1-\sqrt{\mu s}} (f(x_k)-f(x^*))\\
	&\leq \left(\frac{1+\sqrt{\mu s}}{1-\sqrt{\mu s}}\right)\sqrt{s}\langle \nabla f(x_{k+1}), v_k\rangle-\frac{1}{2L}\left(\frac{1+\sqrt{\mu s}}{1-\sqrt{\mu s}}\right)\|\nabla f(x_{k+1})-\nabla f(x_k)\|^2
\end{align*}
For the second part, by using $(\ref{Phase Space Representation II})$, 
\begin{align*}
	\Delta_{\mathbf{II}} &=\frac{1}{4}\|v_{k+1}\|^2-\frac{1}{4}\|v_k\|^2\\
	&=\frac{1}{2}\langle v_{k+1}-v_k, v_{k+1}\rangle-\frac{1}{4}\|v_{k+1}-v_k\|^2\\
	&=-\frac{\sqrt{\mu s}}{1-\sqrt{\mu s}}\|v_{k+1}\|^2-\frac{\beta\sqrt{s}}{2}\langle \nabla f(x_{k+1})-\nabla f(x_k),v_{k+1}\rangle\\
	&-\frac{1+\sqrt{\mu s}}{1-\sqrt{\mu s}}\cdot\frac{\sqrt{s}}{2}\langle \nabla f(x_{k+1}), v_{k+1}\rangle-\frac{1}{4}\|v_{k+1}-v_k\|^2\\
	&=-\frac{\sqrt{\mu s}}{1-\sqrt{\mu s}}\|v_{k+1}\|^2-\frac{\beta\sqrt{s}}{2}\cdot\frac{1-\sqrt{\mu s}}{1+\sqrt{\mu s}}\langle \nabla f(x_{k+1})-\nabla f(x_k),v_{k+1}\rangle\\
	&+\frac{1-\sqrt{\mu s}}{1+\sqrt{\mu s}}\cdot\frac{\beta s}{2}\|\nabla f(x_{k+1})-\nabla f(x_k)\|^2+\frac{\beta s}{2}\langle \nabla f(x_{k+1})-\nabla f(x_k),\nabla f(x_{k+1})\rangle\\
	&-\frac{1+\sqrt{\mu s}}{1-\sqrt{\mu s}}\cdot\frac{\sqrt{s}}{2}\langle \nabla f(x_{k+1}), v_{k+1}\rangle-\frac{1}{4}\|v_{k+1}-v_k\|^2
\end{align*}
For the third part, 
\begin{align*}
	&\Delta_{\mathbf{III}}\\
	&=\frac{1}{4}\left\|v_{k+1}+\frac{2\sqrt{\mu}}{1-\sqrt{\mu s}}(x_{k+1}-x^*)+\beta\sqrt{s}\nabla f(x_{k+1})\right\|^2-\frac{1}{4}\left\|v_{k}+\frac{2\sqrt{\mu}}{1-\sqrt{\mu s}}(x_{k}-x^*)+\beta\sqrt{s}\nabla f(x_{k})\right\|^2\\
	&=\frac{1}{2}\left\langle-\frac{1+\sqrt{\mu s}}{1-\sqrt{\mu s}}\sqrt{s}\nabla f(x_{k+1}), \frac{1+\sqrt{\mu s}}{1-\sqrt{\mu s}}v_{k+1}+\frac{2\sqrt{\mu}}{1-\sqrt{\mu s}}(x_{k+1}-x^*)+\beta\sqrt{s}\nabla f(x_{k+1})\right\rangle\\
	&-\frac{1}{4}\left(\frac{1+\sqrt{\mu s}}{1-\sqrt{\mu s}}\right)^2 s\|\nabla f(x_{k+1})\|^2\\
		&=-\frac{\sqrt{\mu s}}{1-\sqrt{\mu s}}\frac{1+\sqrt{\mu s}}{1-\sqrt{\mu s}}\langle \nabla f(x_{k+1}), x_{k+1}-x^*\rangle-\frac{1}{2}\left(\frac{1+\sqrt{\mu s}}{1-\sqrt{\mu s}}\right)^2\sqrt{s}\langle \nabla f(x_{k+1}),v_{k+1}\rangle\\
		&-\frac{1}{2}\left(\frac{1+\sqrt{\mu s}}{1-\sqrt{\mu s}}\right)\beta s\|\nabla f(x_{k+1})\|^2-\frac{1}{4}\left(\frac{1+\sqrt{\mu s}}{1-\sqrt{\mu s}}\right)^2s\|\nabla f(x_{k+1})\|^2
\end{align*}
The rest of the argument on estimating the difference $\mathcal E_\beta(k+1)-\mathcal E_\beta(k)$ follows the same method as in \cite{HRODE} so here we do not go into further details. 
\end{proof}
\begin{Rem}
		Notice that for the last two terms above, 
		\begin{align*}
			&\ \ \ \ \left[\frac{\beta s}{2}\left(\frac{1+\sqrt{\mu s}}{1-\sqrt{\mu s}}+\frac{1-\sqrt{\mu s}}{1+\sqrt{\mu s}}\right)-\frac{1}{2L}\left(\frac{\beta-\beta\sqrt{\mu s}}{1+\sqrt{\mu s}}+\frac{1+\sqrt{\mu s}}{1-\sqrt{\mu s}} \right)\right]\|\nabla f(x_{k+1})-\nabla f(x_k)\|^2\\
			&\leq\left[\frac{s}{2}\left(\frac{1+\sqrt{\mu s}}{1-\sqrt{\mu s}}\right)+\frac{\beta s}{2}\left(\frac{1-\sqrt{\mu s}}{1+\sqrt{\mu s}}\right)-\frac{\beta}{2L}\left(\frac{1-\sqrt{\mu s}}{1+\sqrt{\mu s}}\right)-\frac{1}{2L}\left(\frac{1+\sqrt{\mu s}}{1-\sqrt{\mu s}}\right)\right]\|\nabla f(x_{k+1})-\nabla f(x_k)\|^2 \\
			&=\frac{1}{2}\left(s-\frac{1}{L}\right)\left[\frac{1+\sqrt{\mu s}}{1-\sqrt{\mu s}}+\beta\cdot\frac{1-\sqrt{\mu s}}{1+\sqrt{\mu s}} \right]\|\nabla f(x_{k+1})-\nabla f(x_k)\|^2\\
		\end{align*}
		Therefore, under the assumption that $s\leq\frac{1}{L}$,
		\begin{align*}
		\mathcal E_\beta(k+1)-\mathcal E_\beta(k) &\leq -\frac{\sqrt{\mu s}}{1-\sqrt{\mu s}}\left(\frac{1+\sqrt{\mu s}}{1-\sqrt{\mu s}}\cdot\langle \nabla f(x_{k+1}), x_{k+1}-x^*\rangle+\|v_{k+1}\|^2\right)\\
		&\ \ \ \ +\frac{1}{2}\left(\frac{1+\sqrt{\mu s}}{1-\sqrt{\mu s}}\right)\cdot s\cdot \frac{(1+\beta)\sqrt{\mu s}+(1-\beta)}{1-\sqrt{\mu s}}\cdot \|\nabla f(x_{k+1})\|^2
		\end{align*}
	\end{Rem}
	\begin{Cor}
		If $s\leq\frac{1}{2L}\leq\frac{1}{L}$ and $f\in \mathcal S_{\mu,L}^1(\mathbb R^n)$ , then we have  
		\begin{align*}
  		&\mathcal E_\beta(k+1)-\mathcal E_\beta(k)\\
  		 &\leq -\sqrt{\mu s}\left\lbrace \frac{1}{(1-\sqrt{\mu s})^2}\left[1-2Ls\cdot\frac{(\beta-\beta^2)\mu s+(3+\beta^2-2\beta)\sqrt{\mu s}+2-2\beta}{2\sqrt{\mu s}}\right](f(x_{k+1})-f(x^*))\right\rbrace\\
			  &-\sqrt{\mu s}\left\lbrace\frac{\sqrt{\mu s}}{(1-\sqrt{\mu s})^2}\left[f(x_{k+1})-f(x^*)-\left(\frac{\beta^2 s\sqrt{\mu s}-(\beta^2-\beta)s}{2\sqrt{\mu s}}\right)\|\nabla f(x_{k+1})\|^2 \right]\right\rbrace\\
			  &-\sqrt{\mu s}\left\lbrace \frac{\mu}{2(1-\sqrt{\mu s})^2}\|x_{k+1}-x^*\|^2+\frac{1}{1-\sqrt{\mu s}}\|v_{k+1}\|^2\right\rbrace  	
	   \end{align*}
   \end{Cor}
	\begin{proof}
		$f\in \mathcal S_{\mu,L}^1(\mathbb R^n)$, together with the inequality
		\begin{align*}
			\begin{cases}
				f(x^*)\geq f(x_{k+1})+\langle \nabla f(x_{k+1}),x^*-x_{k+1}\rangle+\frac{1}{2L}\|\nabla f(x_{k+1})\|^2\\
				f(x^*)\geq f(x_{k+1})+\langle \nabla f(x_{k+1}), x^*-x_{k+1}\rangle+\frac{\mu}{2}\|x_{k+1}-x^*\|^2\\
				1/L\geq 1/2L \geq s
			\end{cases}
		\end{align*}
		We have that 
		\begin{align*}
			&\mathcal E_\beta(k+1)-\mathcal E_\beta(k)\\
			 &\leq -\frac{\sqrt{\mu s}}{1-\sqrt{\mu s}}\bigg[\left(\frac{1+\sqrt{\mu s}}{1-\sqrt{\mu s}}\right)(f(x_{k+1})-f(x^*))+\frac{1}{2L}\left(\frac{\sqrt{\mu s}}{1-\sqrt{\mu s}}\right)\|\nabla f(x_{k+1})\|^2\\
			  & +\frac{\mu}{2}\left(\frac{1}{1-\sqrt{\mu s}}\right)\|x_{k+1}-x^*\|^2-\left(\frac{1}{2}+\frac{1}{2}\beta+\frac{1-\beta}{\sqrt{\mu s}}\right)\left(\frac{1+\sqrt{\mu s}}{1-\sqrt{\mu s}}\right)s\|\nabla f(x_{k+1})\|^2\\
			  &+\|v_{k+1}\|^2 \bigg]\\
			  &\leq -\sqrt{\mu s}\left\lbrace \frac{1}{(1-\sqrt{\mu s})^2}\left[f(x_{k+1})-f(x^*)-\frac{(\beta-\beta^2)\mu s+(3+\beta^2-2\beta)\sqrt{\mu s}+2-2\beta}{2\sqrt{\mu s}}s\|\nabla f(x_{k+1})\|^2\right] \right\rbrace\\
			  &-\sqrt{\mu s}\left\lbrace\frac{\sqrt{\mu s}}{(1-\sqrt{\mu s})^2}\left[f(x_{k+1})-f(x^*)-\left(\frac{\beta^2 s\sqrt{\mu s}-(\beta^2-\beta)s}{2\sqrt{\mu s}}\right)\|\nabla f(x_{k+1})\|^2 \right]\right\rbrace\\
			  &-\sqrt{\mu s}\left\lbrace \frac{\mu}{2(1-\sqrt{\mu s})^2}\|x_{k+1}-x^*\|^2+\frac{1}{1-\sqrt{\mu s}}\|v_{k+1}\|^2\right\rbrace\\
			  &\leq -\sqrt{\mu s}\left\lbrace \frac{1}{(1-\sqrt{\mu s})^2}\left[1-2Ls\cdot\frac{(\beta-\beta^2)\mu s+(3+\beta^2-2\beta)\sqrt{\mu s}+2-2\beta}{2\sqrt{\mu s}}\right](f(x_{k+1})-f(x^*))\right\rbrace\\
			  &-\sqrt{\mu s}\left\lbrace\frac{\sqrt{\mu s}}{(1-\sqrt{\mu s})^2}\left[f(x_{k+1})-f(x^*)-\left(\frac{\beta^2 s\sqrt{\mu s}-(\beta^2-\beta)s}{2\sqrt{\mu s}}\right)\|\nabla f(x_{k+1})\|^2 \right]\right\rbrace\\
			  &-\sqrt{\mu s}\left\lbrace \frac{\mu}{2(1-\sqrt{\mu s})^2}\|x_{k+1}-x^*\|^2+\frac{1}{1-\sqrt{\mu s}}\|v_{k+1}\|^2\right\rbrace
	    \end{align*}
   \end{proof}
  \begin{Lemma}\label{recursive lemma}
  	Let $f\in \mathcal S_{\mu,L}^1(\mathbb R^n)$, $\mu\leq L$. Taking any step size $0<s\leq \frac{1}{4L}$, the discrete-time energy functional with $\lbrace x_k\rbrace_{k=0}^{\infty}$ generated by the discrete method satisfies 
  	\begin{equation*}
  		\mathcal E_\beta(k+1)-\mathcal E_\beta(k)\leq -\sqrt{\mu s}\min\left\lbrace \frac{1}{6},\frac{A_\beta}{B_\beta}\right\rbrace\mathcal E_\beta(k+1)
  	\end{equation*} 
  where $\begin{cases}
  	A_{\beta}=\frac{1}{(1-\sqrt{\mu s})^2}\left[1-2Ls\cdot\frac{(\beta-\beta^2)\mu s+(3+\beta^2-2\beta)\sqrt{\mu s}+2-2\beta}{2\sqrt{\mu s}}\right]\\
  	B_{\beta}=\frac{1}{1-\sqrt{\mu s}}+\frac{\beta^2 Ls}{2}
  \end{cases}$
\end{Lemma}
  \begin{proof}
  	Notice that by the previous lemma, 
  	\begin{align*}
  		&\mathcal E_\beta(k+1)-\mathcal E_\beta(k)\\
  		 &\leq -\sqrt{\mu s}\left\lbrace\underbrace{\frac{1}{(1-\sqrt{\mu s})^2}\left[1-2Ls\cdot\frac{(\beta-\beta^2)\mu s+(3+\beta^2-2\beta)\sqrt{\mu s}+2-2\beta}{2\sqrt{\mu s}}\right]}_{:=A_\beta}(f(x_{k+1})-f(x^*))\right\rbrace\\
			  &-\sqrt{\mu s}\left\lbrace\frac{\sqrt{\mu s}}{(1-\sqrt{\mu s})^2}\left[f(x_{k+1})-f(x^*)-\left(\frac{\beta^2 s\sqrt{\mu s}-(\beta^2-\beta)s}{2\sqrt{\mu s}}\right)\|\nabla f(x_{k+1})\|^2 \right]\right\rbrace\\
			  &-\sqrt{\mu s}\left\lbrace \frac{\mu}{2(1-\sqrt{\mu s})^2}\|x_{k+1}-x^*\|^2+\frac{1}{1-\sqrt{\mu s}}\|v_{k+1}\|^2\right\rbrace  	
	   \end{align*}
	On the other hand, we have 
  \begin{align*}
  		\mathcal E_\beta(k)&\leq\underbrace{\left(\frac{1}{1-\sqrt{\mu s}}+\frac{\beta^2 Ls}{2}\right)}_{:=B_\beta}(f(x_k)-f(x^*))+\frac{1+\sqrt{\mu s}+\mu s}{(1-\sqrt{\mu s})^2}\|v_k\|^2\\
  		&+\frac{3\mu}{(1-\sqrt{\mu s})^2}\|x_{k}-x^*\|^2+\frac{\sqrt{\mu s}}{1-\sqrt{\mu s}}\left[f(x_k)-f(x^*)-\left(\frac{\beta^2 s\sqrt{\mu s}-(\beta^2-\beta )s}{2\sqrt{\mu s}}\right)\|\nabla f(x_k)\|^2 \right]
  	\end{align*}
By comparing the coefficients,
\begin{align*}
	\mathcal E_\beta(k+1)-\mathcal E_\beta(k) &\leq -\sqrt{\mu s}\min\left\lbrace\frac{1}{1-\sqrt{\mu s}},\frac{1}{6},\frac{1-\sqrt{\mu s}}{1+\sqrt{\mu s}+\mu s},\frac{A_\beta}{B_\beta} \right\rbrace\mathcal E_\beta(k+1)\\
	&=-\sqrt{\mu s}\min\left\lbrace \frac{1}{6},\frac{A_\beta}{B_\beta}\right\rbrace\mathcal E_\beta(k+1)
\end{align*}
since $\frac{1}{1-\sqrt{\mu s}}>1>\frac{1}{6}$ and $\frac{1-\sqrt{\mu s}}{1+\sqrt{\mu s}+\mu s}>\frac{2}{7}>\frac{1}{6}$. 
\end{proof}
\begin{Lemma}\label{Comparison Lemma}
	When $\frac{25\mu}{(12L-\mu)^2} \leq s\leq\frac{1}{4L}$, there exists a $\beta_c\in [0,1]$ depending on $\mu,s,L$ such that 
	\begin{equation*}
	    \begin{cases}
			\frac{A_\beta}{B_\beta}\leq \frac{1}{6} &\text{when $0\leq \beta\leq \beta_c$}\\
			\frac{A_\beta}{B_\beta}>\frac{1}{6} &\text{when $\beta_c<\beta\leq 1$}
		\end{cases}
	\end{equation*}
\end{Lemma}
\begin{proof}
	For general $\beta\in [0,1]$,
	\begin{align*}
		\frac{A_\beta}{B_\beta} &=\frac{1-Ls\cdot\frac{(\beta-\beta^2)\mu s+(3+\beta^2-2\beta)\sqrt{\mu s}+2-2\beta }{\sqrt{\mu s}}}{(1-\sqrt{\mu s})^2\left(\frac{1}{1-\sqrt{\mu s}}+\frac{\beta^2 Ls}{2} \right)}\\
		&=\frac{(L\mu s^2-Ls\sqrt{\mu s})\beta^2+(2Ls\sqrt{\mu s}-L\mu s^2+2Ls)\beta+(\sqrt{\mu s}-3Ls\sqrt{\mu s}-2Ls)}{\frac{Ls}{2}\sqrt{\mu s}(1-\sqrt{\mu s})^2\beta^2+\sqrt{\mu s}-\mu s}
	\end{align*}
	To compare $\frac{A_\beta}{B_\beta}$ with $\frac{1}{6}$, we only need to compare the function $h(\beta)$ with $0$ where 
	\begin{align*}
		h(\beta) &=(L\mu s^2-Ls\sqrt{\mu s})\beta^2+(2Ls\sqrt{\mu s}-L\mu s^2+2Ls)\beta+(\sqrt{\mu s}-3Ls\sqrt{\mu s}-2Ls)\\
		&\ \ -\frac{1}{6}\left\lbrace \frac{Ls}{2}\sqrt{\mu s}(1-\sqrt{\mu s})^2\beta^2+\sqrt{\mu s}-\mu s \right\rbrace
	\end{align*}
	First, it is easy to see that 
	\begin{equation*}
		h(0)=\frac{5}{6}\sqrt{\mu s}-3Ls\sqrt{\mu s}+\frac{1}{6}\mu s-2Ls\leq 0
	\end{equation*}
	and 
	\begin{equation*}
		h(1)=\frac{1-2Ls}{1-\sqrt{\mu s}+\frac{Ls}{2}(1-\sqrt{\mu s})^2}\geq 0
	\end{equation*}
	when $\frac{25\mu}{(12L-\mu)^2} \leq s\leq\frac{1}{4L}$. Secondly, 
	\begin{equation*}
		h'(\beta)=Ls\sqrt{\mu s}\left[2(\sqrt{\mu s}-1)-\frac{1}{6}(1-\sqrt{\mu s})^2 \right]\beta+2Ls\sqrt{\mu s}-L\mu s^2+2Ls
	\end{equation*}
	which is a monotone decreasing function on $[0,1]$. Hence,
	\begin{equation*}
		h'(\beta)\geq h'(1)=L\mu s^2+2Ls-\frac{1}{6}Ls\sqrt{\mu s}(1-\sqrt{\mu s})^2\geq L\mu s^2+2Ls-\frac{1}{12}Ls\geq 0
	\end{equation*} 
	Therefore, $h'(\beta)\geq 0$ for all $0\leq \beta\leq 1$. This completes the proof. 
\end{proof}
\begin{Rem}\label{Critical Value}
	The $\beta_c$ in the Lemma above is computable, 
	\begin{align*}
		\beta_c=\frac{-\mathbf{b}-\sqrt{\mathbf{b}^2-4\mathbf a\mathbf c}}{2\mathbf a}
	\end{align*}
	where 
	\begin{align*}
		\mathbf a &=Ls\sqrt{\mu s}(\sqrt{\mu s}-1)\left[1-\frac{1}{12}(\sqrt{\mu s}-1)\right]\\
		\mathbf b &=Ls(2\sqrt{\mu s}-\mu s+2)\\
		\mathbf c &=\frac{5}{6}\sqrt{\mu s}-3Ls\sqrt{\mu s}-2Ls+\frac{1}{6}\mu s
	\end{align*}
\end{Rem}
\begin{Cor}\label{Inductive difference of energy functional}
	Suppose $\frac{25\mu}{(12L-\mu)^2}\leq s\leq\frac{1}{4L}$. When $0\leq \beta\leq \beta_c$,
	\begin{align*}
		\mathcal E_\beta(k+1)-\mathcal E_\beta(k)\leq -\frac{\sqrt{\mu s}-Ls\left[(\beta-\beta^2)\mu s+(3+\beta^2-2\beta)\sqrt{\mu s}+2-2\beta\right]}{\sqrt{\mu s}\left(1-\sqrt{\mu s}+\frac{\beta^2 Ls}{2}(1-\sqrt{\mu s})^2\right)}\mathcal E_\beta(k)
	\end{align*}
	When $\beta_c\leq \beta\leq 1$, 
	\begin{equation*}
		\mathcal E_\beta(k+1)-\mathcal E_\beta(k)\leq -\frac{\sqrt{\mu s}}{6}\mathcal E_\beta(k)
	\end{equation*}
	$(2)$
\end{Cor}
\begin{proof}
		Trivially from Lemma $\ref{Comparison Lemma}$. 
\end{proof}
\subsection{Proof of Main Results}
\begin{proof}[Proof of Theorem \ref{Main Theorem}]
	Notice that 
	\begin{align*}
		\mathcal E_\beta(k)\geq \frac{1+\sqrt{\mu/(cL)}}{1-\sqrt{\mu/(cL)}}\left(f(x_k)-f(x^*) \right)-\frac{\beta\|\nabla f(x_k)\|^2}{(2cL)(1-\sqrt{\mu/(cL)})}
	\end{align*}
	Together with 
	\begin{align*}
		f(x_k)-f(x^*)\geq \frac{1}{2L}\|\nabla f(x_k)\|^2
	\end{align*}
	we get 
	\begin{align*}
		\mathcal E_\beta(k)\geq \frac{1+\sqrt{\mu/(cL)}}{1-\sqrt{\mu/(cL)}}\left(f(x_k)-f(x^*) \right)-\frac{\beta\left(f(x_k)-f(x^*)\right)}{c(1-\sqrt{\mu/(cL)})}
	\end{align*}
	Equivalently, 
	\begin{align*}
		f(x_k)-f(x^*)\leq \frac{c+c\sqrt{\mu/(cL)}-\beta}{c\left(1-\sqrt{\mu/(cL)} \right)} \cdot\mathcal E_\beta(k)
	\end{align*}
	Applying Corollary $\ref{Inductive difference of energy functional}$ inductively and plugging in $s=\frac{1}{cL}$ gives us 
	\begin{equation*}
		\mathcal E_\beta(k)\leq\frac{\mathcal E_\beta(0)}{\left\lbrace 1+\frac{\sqrt{\frac{\mu}{cL}}-\frac{1}{c}\left[(\beta-\beta^2)\frac{\mu}{cL}+(3+\beta^2-2\beta)\sqrt{\frac{\mu}{cL}}+2-2\beta \right]}{\sqrt{\frac{\mu}{cL}}\left[1-\sqrt{\frac{\mu}{cL}}+\frac{\beta^2}{2c}\left(1-\sqrt{\frac{\mu}{cL}} \right)^2 \right]} \right\rbrace^k}
	\end{equation*} 
	Recall that the initial velocity $v_0=-\frac{2\sqrt{s}\nabla f(x_0)}{1+\sqrt{\mu s}}$, hence 
	\begin{align*}
	 \mathcal E_\beta(0)\leq &\frac{1+\sqrt{\mu s}}{1-\sqrt{\mu s}}\left(f(x_0)-f(x^*)\right)+\frac{s}{(1+\sqrt{\mu s})^2}\|\nabla f(x_0)\|^2\\
	 &+\frac{1}{4}\left\|\frac{2\sqrt{\mu}}{1-\sqrt{\mu s}}(x_0-x^*)-\left(\frac{2-\beta-\beta\sqrt{\mu s}}{1+\sqrt{\mu s}}\right)\cdot\sqrt{s}\cdot\nabla f(x_0)\right\|^2\\
	 &\leq \left[\frac{1}{2}\left(\frac{1+\sqrt{\mu s}}{1-\sqrt{\mu s}}\right)+\frac{Ls}{(1+\sqrt{\mu s})^2}+\frac{2\mu/L}{(1-\sqrt{\mu s})^2}+\frac{Ls}{2}\left(\frac{2-\beta-\beta\sqrt{\mu s}}{1+\sqrt{\mu s}}\right)^2\right]L\|x_0-x^*\|^2\\
	 &=C_{\beta,\mu,L}\cdot L\cdot \|x_0-x^*\|^2
	\end{align*}
	where 
	\begin{align*}
		C_{\beta,\mu,L}=\left[\frac{1+\sqrt{\mu/(sL)}}{2-2\sqrt{\mu/(sL)}}+\frac{1}{4(1+\sqrt{\mu/(sL)})^2}+\frac{2\mu/L}{(1-\sqrt{\mu/(sL)})^2}+\frac{1}{2c}\left(\frac{2-\beta-\beta\sqrt{\mu s}}{1+\sqrt{\mu s}}\right)^2 \right]
	\end{align*}
	since we write $s=\frac{1}{cL}$. Let 
	\begin{equation*}
		C'_{\beta,\mu,L}=\frac{c+c\sqrt{\mu/(cL)}-\beta}{c\left(1-\sqrt{\mu/(cL)} \right)}\cdot C_{\beta,\mu,L}
	\end{equation*}
	and we conclude that 
	\begin{align*}
		f(x_k)-f(x^*)&\leq \frac{C'_{\beta,\mu,L}\cdot L\cdot \|x_0-x^*\|^2}{\left\lbrace 1+\frac{\sqrt{\frac{\mu}{cL}}-\frac{1}{c}\left[(\beta-\beta^2)\frac{\mu}{cL}+(3+\beta^2-2\beta)\sqrt{\frac{\mu}{cL}}+2-2\beta \right]}{\sqrt{\frac{\mu}{cL}}\left[1-\sqrt{\frac{\mu}{cL}}+\frac{\beta^2}{2c}\left(1-\sqrt{\frac{\mu}{cL}} \right)^2 \right]} \right\rbrace^k}\\
		&=\frac{C'_{\beta,\mu,L}\cdot L\cdot\|x_0-x^*\|^2 }{\left\lbrace 1+\frac{\frac{\beta^2-\beta}{c^2}\left(\frac{\mu}{L}\right)+\left(\frac{1}{\sqrt{c}}-\frac{3+\beta^2-2\beta}{c\sqrt{c}}\right)\sqrt{\frac{\mu}{L}}-\frac{2-2\beta}{c}}{\frac{\beta^2}{2c^2\sqrt{c}}\left(\frac{\mu}{L}\right)^{3/2}-\left(\frac{1}{c}+\frac{\beta^2}{c^2}\right)\frac{\mu}{L}+\left(\frac{1}{\sqrt{c}}+\frac{\beta^2}{2c\sqrt{c}}\right)\sqrt{\frac{\mu}{L}} }\right\rbrace^k}
	\end{align*}
	This completes the proof of the subcritical regime. The supercritical regime follows directly from Corollary $\ref{Inductive difference of energy functional}$ and \cite{HRODE} Theorem $3$. 
\end{proof}
\begin{Rem}
	As $\beta$ reaches to $0$, since $\frac{A_0}{B_0}\leq\frac{1}{6}$ if $\frac{25\mu}{(12L-\mu)^2} \leq s\leq \frac{1}{4L}$, we have to choose a step size smaller than $\frac{25\mu}{(12L-\mu)^2}$ in order to let $\frac{A_0}{B_0}>\frac{1}{6}$. For instance, $\mu=\frac{\mu}{16L^2}$ works here. This matches the $(\ref{rate of convergence for heavy ball method})$ and gives a different reasoning than the one stated in \cite{HRODE} of why we need a more conservative step size on Heavy ball method.
\end{Rem}

\end{document}